\documentclass[11pt,leqno]{article}

\usepackage{amsmath,amssymb,amsthm}
\usepackage[all]{xy}
\usepackage{lscape}

\makeatletter

\newtheorem*{Theorem}{Theorem~\ref{main.thm}}

\newtheorem*{Corollary}{Corollary~\ref{main.cor}}

\newtheorem{theorem}{Theorem}[section]

\newtheorem{lemma}[theorem]{Lemma}
\newtheorem{corollary}[theorem]{Corollary}

\theoremstyle{definition}

\theoremstyle{remark}

\theoremstyle{remark}

\def\({{\rm (}}
\def\){{\rm )}}

\let\Mathrm\operator@font

\let\Bbb\mathbb
\newcommand{\fm}{\ensuremath{\mathfrak m}}
\newcommand{\fn}{\ensuremath{\mathfrak n}}

\def\standop#1{\mathop{\Mathrm #1}\nolimits}
\def\difstop#1#2{\expandafter\def\csname #1\endcsname{\standop{#2}}}
\def\defstop#1{\difstop{#1}{#1}}

\defstop{AB}
\defstop{ann}
\defstop{Ass}
\defstop{Add}
\defstop{Alt}
\defstop{Ass}
\difstop{adim}{AFHdim}

\defstop{Bl}

\defstop{CMFI}
\defstop{codim}
\defstop{Coh}
\defstop{coht}
\defstop{Coker}
\defstop{Cone}
\defstop{Cl}
\defstop{Cox}
\defstop{cosk}
\defstop{cd}
\defstop{cmd}
\difstop{charac}{char}

\defstop{depth}
\defstop{Der}
\defstop{Div}
\defstop{div}

\defstop{EM}
\defstop{embdim}
\defstop{End}
\defstop{ev}
\defstop{Ext}

\difstop{fdim}{flat.dim}
\defstop{Flat}
\defstop{Func}
\defstop{Fpqc}

\defstop{GD}
\defstop{Good}
\defstop{Gal}
\defstop{Grass}

\difstop{height}{ht}
\defstop{Hom}
\defstop{Hy}

\difstop{Image}{Im}
\defstop{ind}
\defstop{ini}

\defstop{Ker}

\defstop{Lch}
\defstop{length}
\defstop{Lin}
\defstop{Lqc}
\defstop{lqc}
\defstop{LQ}
\defstop{LM}
\defstop{Lie}
\defstop{Loc}

\defstop{Mat}
\defstop{Max}
\defstop{Min}
\defstop{Mod}
\defstop{Mor}
\defstop{MCM}
\defstop{Map}
\difstop{mred}{red}

\defstop{Nerve}
\defstop{NonSFR}
\defstop{NonCMFI}
\defstop{NonNor}
\defstop{Nor}

\defstop{ob}
\defstop{Ob}
\def\op{^{\standop{op}}}

\defstop{PA}
\defstop{PM}
\defstop{PR}
\defstop{Proj}
\defstop{Prin}
\defstop{Pic}

\defstop{Qch}
\defstop{qch}

\defstop{rad}
\defstop{rank}

\defstop{res}
\defstop{Reg}
\defstop{Ref}
\defstop{Rat}

\defstop{Spec}
\defstop{supp}
\defstop{Supp}
\defstop{Sym}
\defstop{Sing}
\defstop{SFR}
\defstop{Soc}
\defstop{Sh}

\difstop{tdeg}{trans.deg}

\defstop{Tor}
\defstop{TRD}
\difstop{trace}{tr}
\defstop{tot}

\defstop{Zar}
\defstop{FL}
\defstop{add}
\defstop{Ind}
\defstop{grmod}

\difstop{summ}{sum}

\let\frak\mathfrak
\def\({{\rm(}}
\def\){{\rm)}}

\def\fm{\mathfrak{m}}

\def\QQ{{\mathbb{Q}}}

\def\ZZ{{\mathbb{Z}}}

\let\projlim\varprojlim

\def\sdarrow#1{\downarrow\hbox to 0pt{\scriptsize$#1$\hss}}
\def\suarrow#1{\uparrow\hbox to 0pt{\scriptsize$#1$\hss}}
\def\ssearrow#1{\searrow\hbox to 0pt{\scriptsize$#1$\hss}}

\def\section{\@startsection{section}{1}{\z@ }%
  {-3.5ex plus -1ex minus -.2ex}{2.3ex plus .2ex}{\bf }}

\long\def\refname{\par\kern -3ex
  \begin{center}\rm R\sc{eferences}\end{center}\par\kern 
  -2ex}

\def\@seccntformat#1{\csname the#1\endcsname.\quad}

\def\@@@sect#1#2#3#4#5#6[#7]#8{%
  \ifnum #2>\c@secnumdepth 
  \def \@svsec {}\else \refstepcounter {#1}%
  \def\@svsec{}
  \fi 
  \@tempskipa #5\relax 
  \ifdim \@tempskipa >\z@ 
  \begingroup #6\relax \@hangfrom {\hskip #3\relax 
    \@svsec}{\interlinepenalty \@M #8\par }\endgroup 
  \csname #1mark\endcsname {#7}
  \else 
  \def \@svsechd {#6\hskip #3\@svsec #8\csname #1mark\endcsname {#7}}
  \fi \@xsect {#5}}

\def\@@@startsection#1#2#3#4#5#6{%
  \if@noskipsec \leavevmode \fi \par \@tempskipa #4\relax \@afterindenttrue 
  \ifdim \@tempskipa <\z@ \@tempskipa -\@tempskipa \@afterindentfalse 
  \fi \if@nobreak \everypar {}\else \addpenalty {\@secpenalty }\addvspace 
  {\@tempskipa }\fi \@ifstar {\@ssect {#3}{#4}{#5}{#6}}{\@dblarg 
    {\@@@sect {#1}{#2}{#3}{#4}{#5}{#6}}}}

\def\theparagraph{\thesection.\arabic{paragraph}}
\def\aparagraph{\@@@startsection{paragraph}{2}{\z@ }%
  {1.75ex plus .2ex minus .15ex}{-1em}{\bf(\theparagraph) } }
\def\paragraph{\@@@startsection{paragraph}{2}{\z@ }%
  {1.75ex plus .2ex minus .15ex}{-1em}{}{\bf(\theparagraph)} }

\let\c@theorem\c@paragraph

\advance\textheight 16mm
\advance\textwidth 20mm
\title{Indecomposability of graded modules\\over a graded ring}
\author{M{\sc itsuyasu} H{\sc ashimoto}\thanks{Partially supported by JSPS KAKENHI Grant number 20K03538
    and MEXT Promotion of Distinctive Joint Research Center Program JPMXP0619217849.}
  \and
  Y{\sc untian} Y{\sc ang}}
\date{}

\begin{document}

\maketitle
\footnote[0]
{2020 \textit{Mathematics Subject Classification}. 
  Primary 16W50; Secondary 13A02.
  Key Words and Phrases.
  graded ring, indecomposable module
}

\begin{abstract}
  Let $R=\bigoplus_{i\geq 0}R_i$ be a Noetherian commutative non-negatively graded ring such that $(R_0,\fm_0)$ is a
  Henselian local ring.
  Let $\fm$ be its unique graded maximal ideal $\fm_0+\bigoplus_{i>0}R_i$.
  Let $T$ be a module-finite (non-commutative) graded $R$-algebra.
  Let $T\grmod$ denote the category of finite graded left $T$-modules,
  and $M\in T\grmod$.
  Then the following are equivalent:
  (1) $\hat M$ is an indecomposable $\hat T$-module, where $\widehat{(-)}$ denotes the $\fm$-adic completion;
  (2) $M_\fm$ is an indecomposable $T_\fm$-module;
  (3) $M$ is an indecomposable $T$-module;
  (4) $M$ is indecomposable as a graded $T$-module.
  As a corollary we prove that for two finite graded left $T$-modules $M$ and $N$, the following are equivalent:
  (1) If $M=M_1\oplus\cdots\oplus M_s$ and $N=N_1\oplus\cdots\oplus N_t$ are
    decompositions into indecomposable objects in $T\grmod$, then $s=t$, and
    there exist some permutation $\sigma\in \frak S_s$ and integers $d_1,\ldots,d_s$
    such that $N_i\cong M_{\sigma i}(d_i)$, where $-(d_i)$ denotes the shift of degree;
  (2) $M\cong N$ as $T$-modules;
  (3) $M_\fm\cong N_\fm$ as $T_\fm$-modules;
    (4) $\hat M\cong \hat N$ as $\hat T$-modules.
    As an application, we compare the FFRT property of rings of characteristic $p$ in the
    graded sense and in the local sense.
\end{abstract}

\section{Introduction}

In many cases, a theory in Noetherian local rings has its graded version
or ${}^*$version.
Moreover, such a graded version often determines the ring-theoretic property of the graded ring $A$
as a non-graded ring.

We give an easy illustrative example.
Let $A=\bigoplus_{i\geq 0}$ be a finitely generated positively graded (that is, $A_0=k$) commutative algebra over a field $k$.
Then we can define a homogeneous system of parameters of $A$ to be a sequence $x_1,\ldots,x_d$ of
homogeneous elements of positive degrees such that $\dim_k A/(x_1,\ldots,x_d)<\infty$, and
the length of the sequence $d$ is small as possible among such sequences.
It is well-known that a sequence $x_1,\ldots,x_r$ of homogeneous elements of positive degree is a homogeneous system of
parameters if and only if their images in the local ring $A_\fm$ is a system of parameters in the local sense.
In particular, if $x_1,\ldots,x_d$ is a homogeneous system of parameters, then we have $d=\dim A_\fm$.
What is interesting is $\dim A=\dim A_\fm$ in fact \cite[Section~1.5]{BH}.
Moreover, there is a homogeneous system of parameters which is a regular sequence if and only if
$A_\fm$ is Cohen--Macaulay if and only if $A$ is Cohen--Macaulay \cite{MR}.

Some important results in this direction can be found in the textbook \cite{BH}.

In this paper, we compare the indecomposability of finitely generated module $M$ over a module-finite
algebra over a graded ring along this line, and we get the following.
Namely, we have

\begin{Theorem}
  Let $R=\bigoplus_{i\geq 0}R_i$ be a Noetherian commutative graded ring such that $(R_0,\fm_0)$ is Henselian local,
  $R_+=\bigoplus_{i>0}R_i$, $\fm=R_++\fm_0$, and $T$ a $\Bbb Z$-graded module-finite \(non-commutative\) $R$-algebra.
  Let $M$ be a finitely generated graded left $T$-module.
  Then the following are equivalent.
  \begin{enumerate}
  \item[\rm(1)] $\hat M$ is indecomposable as a $\hat T$-module, where $\widehat{(-)}$ denotes
    the $\fm$-adic completion.
  \item[\rm(2)] $M_\fm$ is indecomposable as a $T_\fm$-module.
  \item[\rm(3)] $M$ is indecomposable as a $T$-module.
  \item[\rm(4)] $M$ is indecomposable as a graded $T$-module.
  \end{enumerate}
\end{Theorem}

As a corollary, we prove

\begin{Corollary}
  Let $M$ and $N$ be objects in $T\grmod$, the category of finitely generated graded left $T$-modules.
  Then the following are equivalent.
  \begin{enumerate}
  \item[\rm(1)] If $M=M_1\oplus\cdots\oplus M_s$ and $N=N_1\oplus\cdots\oplus N_t$ are
    decompositions into indecomposable objects in $T\grmod$, then $s=t$, and
    there exist some permutation $\sigma\in \frak S_s$ and integers $d_1,\ldots,d_s$
    such that $N_i\cong M_{\sigma i}(d_i)$, where $-(d_i)$ denotes the shift of degree.
  \item[\rm(2)] $M\cong N$ as $T$-modules.
  \item[\rm(3)] $M_\fm\cong N_\fm$ as $T_\fm$-modules.
  \item[\rm(4)] $\hat M\cong \hat N$ as $\hat T$-modules.
  \end{enumerate}
\end{Corollary}

As an application, we prove a comparison theorem for the finite $F$-representation type (FFRT for short) property
of an $F$-finite non-negatively graded ring $R=\bigoplus_{i\geq 0}R_i$ such that $(R_0,\fm_0)$ is
an $F$-finite Henselian local ring of a prime characteristic $p$ in the graded sense
and the FFRT property of the $\fm$-adic completion of $R$, where $\fm=\fm_0+\bigoplus_{i>0}R_i$
(Corollary~\ref{main2.cor}).
See Corollary~\ref{main2.cor} for the definition of FFRT.
This property for rings of characteristic $p$ was defined by K.~E.~Smith and M.~Van den Bergh \cite{SvdB},
and has been studied extensively \cite{AK,DQ,HB,HO,Shibuta1,Shibuta2}.

\medskip

Acknowledgment.
The essential part of this joint work was done in 2019, and is recorded as
\cite{Yang}.
In \cite{Yang}, Theorem~\ref{main.thm} and Corollary~\ref{main.cor} for the
case that $R=T$ was proved.
The authors are grateful to Professor Osamu Iyama for valuable discussions.

\section{Preliminaries}

\paragraph
For a ring $A$, we denote the set of units of $A$ by $A^\times$.
We say that $A$ is {\em local} if $A\setminus A^\times$ is an additive subgroup of $A$.
This is equivalent to say that $A\neq 0$, and $A\setminus A^\times$ is closed under addition.
If so, $A\setminus A^\times$ is the unique maximal left ideal of $A$.
It is also the unique maximal right ideal of $A$, and agrees with the radical $\rad A$.
Note that $A$ is local if and only if $A/\rad A$ is a division ring, see \cite[(19.1)]{Lam}.

\paragraph
Let $B=\bigoplus_{i\in\ZZ}B_i$ be a graded ring, and $M=\bigoplus_{i\in\ZZ}M_i$ be a graded left module.
A left graded submodule $N$ is said to be ${}^*$maximal if $N$ is a maximal element of the
set $\{N\subset M\mid \text{$N$ is a graded submodule of $M$ and $N\neq M$}\}$.
The intersection of all the ${}^*$maximal graded submodules is denoted by ${}^*\rad M$.
If $f:M\rightarrow M'$ is a homomorphism of graded left $B$-modules, then $f({}^*\rad M)\subset{}^*\rad M'$.
In particular, if $b\in B_i$, then $({}^*\rad B)b\subset{}^*\rad(B(i))=({}^*\rad B)(i)$,
and ${}^*\rad B$ is a two-sided ideal.

\begin{lemma}\label{*maximal-exists.lem}
  Let $B$ be as above.
  Then any nonzero finitely generated left graded module $M$ has a ${}^*$maximal submodule.
  In particular, $M\neq {}^*\rad M$.
\end{lemma}

\begin{proof}
  Let $\Gamma$ be the set of graded proper submodules of $M$.
  As $M$ is nonzero, $\Gamma$ is nonempty.
  If $\Omega$ is a non-empty chain of $\Gamma$, then $\sum_{N\in\Omega}N\neq M$, and $\sum_{N\in\Omega}N\in\Gamma$.
  Indeed, if $\sum_{N\in\Omega}N=M$, then there exists some $N$ that contains all the
  generators of $M$.
  This implies $N=M\in\Omega\subset\Gamma$ and this is a contradiction.
  So by Zorn's lemma, $\Gamma$ has a maximal element, and this is what we wanted to prove.
\end{proof}

\begin{corollary}\label{*maximal-exists.cor}
  Let $B$ be as above, and assume that $B\neq 0$.
  Then $B$ has a left ${}^*$maximal ideal.
  In particular, $B\neq{}^*\rad B$.
  \qed
\end{corollary}

\begin{lemma}
  Let $B$ be as above, $i\in\ZZ$, and $b\in B_i$.
  Then the following are equivalent.
  \begin{enumerate}
  \item[\rm(1)] For any $c\in B_{-i}$, $1+cb\in B_0^\times$.
  \item[\rm(2)] For any $c\in B_{-i}$, $1+cb\in B^\times$.
  \item[\rm(3)] $b\in{}^*\rad B$.
    \item[\rm(4)] $b\in {}^*\rad B\op$, where $B\op$ is the opposite ring of $B$.
    \item[\rm(5)] For any $c\in B_{-i}$, $1+bc\in B^\times$.
    \item[\rm(6)] For any $c\in B_{-i}$, $1+bc\in B_0^\times$.
  \end{enumerate}
\end{lemma}

\begin{proof}
  It is obvious that $B_0^\times\subset B_0\cap B^\times$.
  Conversely, if $x\in B_0\cap B^\times$ and $xy=1$ with $y=\sum y_j$ ($y_j\in B_j$),
  then $xy_0=1$, and hence $x\in B_0^\times$.
  So (1)$\Leftrightarrow$(2) and (5)$\Leftrightarrow$(6) are obvious.

  If $b\notin{}^*\rad B$, then $b\notin \fm$ for some ${}^*$maximal left ideal $\fm$ of $B$.
  So $B=Bb+\fm$, and $1+cb\in\fm$ for some $c\in B_{-i}$.
  Thus we have (2)$\Rightarrow$(3).
  Similarly, (5)$\Rightarrow$(4) is proved.
  
  (3)$\Rightarrow$(2).
If $b\in{}^*\rad B$ and $c\in B_{-i}$, then $cb\in(\rad B)_0$.
  So $1+cb\in B_0$ is not contained in any ${}^*$maximal ideal of $B$, and hence
  $B(1+cb)=B$.
  So there exists some $d\in B_0$ such that $d(1+cb)=1$.
  If $d\in\fm$ for some ${}^*$maximal ideal $\fm$ of $B$, then $d+dcb\in\fm$, 
  since $dcb\in{}^*\rad B$.
  This shows that $1=d(1+cb)\in\fm$, which is absurd.
  So $d\in B^\times$, and hence $1+cb\in B^\times$.

  (3)$\Rightarrow$(5).
  As $b\in{}^*\rad B$ and ${}^*\rad B$ is a two-sided ideal, $bc\in{}^*\rad B$.
  So by the assertion (3)$\Rightarrow$(2), which we have already proved, we have
  that $1+bc\in B^\times$.

  (4)$\Rightarrow$(2) follows from (3)$\Rightarrow$(5) above, applied to the
  graded ring $B\op$.
\end{proof}

\begin{lemma}\label{*local.lem}
  Let $B$ be as above.
  Then the following are equivalent.
  \begin{enumerate}
  \item[\rm(1)] $B_0$ is local.
  \item[\rm(2)] $B$ has a unique ${}^*$maximal left ideal.
  \item[\rm(3)] $B$ has a unique ${}^*$maximal right ideal.
  \item[\rm(4)] For each $i\in\ZZ$, $B_i\setminus B^\times$ is an additive subgroup of $B_i$.
  \end{enumerate}
\end{lemma}

\begin{proof}
  (1)$\Rightarrow$(2).
  If $B=0$, then $B_0=0$, and $B_0$ is not local.
  So $B\neq 0$, and $B$ has a ${}^*$maximal left ideal by Lemma~\ref{*maximal-exists.lem}.
  If $B$ has two ${}^*$maximal left ideals $\fm_1$ and $\fm_2$ with $\fm_1\neq\fm_2$,
  then there exists some $a_1\in\fm_1\cap B_0$ and $a_2\in\fm_2\cap B_0$ such that
  $a_1+a_2=1$, and $B_0$ is not local.

  (2)$\Rightarrow$(4).
  Let $\fm$ be the unique ${}^*$maximal left ideal.
  If $b\in B_i\setminus B^\times$, then $b\in Bb\subset\fm$, and
  $B_i\setminus B^\times=B_i\cap \fm$.

  (4)$\Rightarrow$(1).
  This is trivial.

  (1)$\Leftrightarrow$(3) follows from (1)$\Leftrightarrow$(2), which already has been proved,
  applied to $B\op$.
\end{proof}

\paragraph
We say that $B$ is ${}^*$local if $B$ satisfies the equivalent conditions in Lemma~\ref{*local.lem}.

\begin{lemma}[graded Nakayama's lemma]\label{*NAK.lem}
  Let $B$ be as above, and $J={}^*\rad B$.
  If $M$ is a finitely generated graded left $B$-module and $JM=M$, then $M=0$.
\end{lemma}

\begin{proof}
  For any homogeneous element $m$ of $M$, $x\mapsto xm$ is a graded homomorphism of left $B$-modules
  $B\rightarrow M(-\deg m)$, and ${}^*\rad B$ is mapped to ${}^*\rad M$ by this map.
  This shows that $JM\subset {}^*\rad M\subset M$.
  By assumption, we have that ${}^*\rad M=M$.
  By Lemma~\ref{*maximal-exists.lem}, we have that $M=0$.
\end{proof}

\section{Main results}

\begin{lemma}\label{semi-perfect.lem}
  Let $(A,\fm)$ be a Henselian local ring, and $\Lambda$ a module-finite \(non-commutative\) $A$-algebra.
  Then $\Lambda$ is semi-perfect, and the category of finite left $\Lambda$-modules is Krull--Schmidt.
\end{lemma}

For Henselian local rings, see \cite[Section~A.3]{LW} and \cite{Milne}.

\begin{proof}
  Let $V$ be a simple left $\Lambda$-module.
  Let $\bar \Lambda=\Lambda/\rad\Lambda$.
  As $\fm\Lambda\subset\rad\Lambda$ (by Nakayama's lemma applied to the finite $A$-module $\Lambda$),
  we have that $\bar\Lambda$ is a finite-dimensional $A/\fm$-algebra
  with the trivial radical.
  So $\bar\Lambda$ is semi-simple.
  As $V$ is a simple $\bar\Lambda$-module, there exists some idempotent $\bar e$ of $\bar\Lambda$ such that
  $V=\bar\Lambda\bar e$.
  By \cite[Theorem~A.30]{LW}, $\bar e$ lifts to an idempotent $e$ of $\Lambda$.
  Let $P=\Lambda e$.
  Then $P/JP=(\Lambda/J)e=\bar \Lambda\bar e=V$.
  By Nakayama's lemma, $JP$ is the unique maximal submodule of $P$, and hence
  the canonical surjective map $P\rightarrow P/JP\cong V$ is a projective cover
  by \cite[Lemma~3.6]{Krause}.
  So $\Lambda$ is semi-perfect, and the category of left $\Lambda$-modules is Krull--Schmidt,
  see \cite[Proposition~4.1]{Krause}.
\end{proof}

\paragraph
Let $R=\bigoplus_{i>0}R_i$ be a Noetherian $\ZZ_{\geq 0}$-graded commutative ring such that
$(R_0,\fm_0)$ is Henselian local.
We set $\fm=R_++\fm_0$, where $R_+=\bigoplus_{i>0}R_i$.

\begin{lemma}\label{Krull.lem}
  Let $M$ be a finite graded $R$-module.
  Then $\bigcap_{r\geq 1}\fm^rM=0$.
\end{lemma}

\begin{proof}
  Let $N=\bigcap_{r\geq 1}\fm^rM$.
  Note that $N$ is a finite graded $R$-submodule of $M$.
  By Artin--Rees lemma, there exists some $c\geq 1$ such that
  \[
  N=N\cap \fm^{c+1}M=\fm(N\cap \fm^cM)\subset \fm N\subset N.
  \]
  As $\fm N=N$, $N=0$ by graded Nakayama's lemma.
\end{proof}

Let $T=\bigoplus_{i\in\ZZ}T_i$ be a graded $R$-algebra which is a finite $R$-module.
Let $J$ be the ${}^*$radical ${}^*\rad T$ of $T$.

\begin{lemma}\label{J^r.lem}
  $J^r\subset \fm T\subset J$ for some $r\geq 1$.
\end{lemma}

\begin{proof}
  If $\fm T\not\subset J$, then there exists 
  some ${}^*$maximal left ideal $\fn$ of $T$ such that $\fn+\fm T=T$.
  By graded Nakayama's lemma, $\fn=T$, and this is absurd.
  So $\fm T\subset J$.

  To prove that $J^r\subset \fm T$ for some $r\geq 1$, we may assume that $\fm=0$.
  Then $T$ is a finite-dimensional $R_0/\fm_0$-algebra.
  As $J^r=J^{r+1}$ for some $r\geq 1$, $J^r=0$ by graded Nakayama's lemma again.
\end{proof}

\begin{lemma}\label{idempotent.lem}
  Let $e$ be an idempotent of $\hat T$, where $\hat T$ is the $\fm$-adic completion of $T$.
  If $e\in J\hat T=\hat J$, then $e=0$.
  In particular, if $e$ is an idempotent of $T$ such that $e\in J$, then $e=0$.
\end{lemma}

\begin{proof}
  By Lemma~\ref{J^r.lem}, we can take $r\geq 1$ such that $J^r\subset \fm T$.
  So $e=e^r\in \fm \hat T$.
  As $e=e^n\in \fm^n \hat T$ for any $n\geq 1$, so $e\in \bigcap_{n\geq 1}\fm^n \hat T=0$.
  Since $T\rightarrow \hat T$ is injective, the last assertion follows immediately.
\end{proof}

\begin{lemma}\label{T-rad.lem}
  If $T$ is ${}^*$local, then $J=\rad T_0+\bigoplus_{i\neq 0}T_i$,
  and $T/J\cong T_0/\rad T_0$ is a division ring.
\end{lemma}

\begin{proof}
  Replacing $T$ by $T/J$, we may assume that $J=0$.
  Then $\fm T\subset J=0$.
  Replacing $R$ by $R/\fm$, we may assume that $R=R_0$ is a field concentrated in degree zero.
  Then $T$ is a finite-dimensional $R$-algebra.
  If $i\neq 0$, $a\in R_i$, and $a$ is a unit, then $R_{ir}\neq 0$ for $r\in\ZZ$,
  and $R$ cannot be finite-dimensional.
  So $T_i\subset{}^*\rad T$ for $i\neq 0$.
  On the other hand, we have
\[
({}^*\rad T)_0=T_0\setminus T^\times=T_0\setminus T_0^\times=\rad T_0.
\]
  So ${}^*\rad T=\rad T_0+\bigoplus_{i\neq 0}T_i$.
  Hence $T/{}^*\rad T\cong T_0/\rad T_0$.
  As $T_0$ is local, this is a division ring.
\end{proof}

\paragraph
Let $T\grmod$ denote the category of finitely generated graded left $T$-modules.
We say that $M\in T\grmod$ is ${}^*$indecomposable if it is an indecomposable object of
$T\grmod$.
For $M\in\grmod T$, the endomorphism ring of $M$ as an object of $T\grmod$ is
$E_0=(\End_TM)_0$, the degree zero component of $E=\End_TM$, the endomorphism ring of $M$ as
a (non-graded) $T$-module.
Note that $E_0$ is finite as an $R_0$-module.
As we assume that $R_0$ is Henselian local, $E_0$ is semi-perfect, and hence
the additive category $T\grmod$ is Krull--Schmidt by Lemma~\ref{semi-perfect.lem}.
In particular, $M$ is ${}^*$indecomposable if and only if $E$ is ${}^*$local, that is,
$E_0$ is local.

\begin{theorem}\label{main.thm}
  Let $R=\bigoplus_{i\geq 0}R_i$ be a Noetherian commutative graded ring such that $(R_0,\fm_0)$ is Henselian local,
  $R_+=\bigoplus_{i>0}R_i$, $\fm=R_++\fm_0$, and $T$ a $\Bbb Z$-graded module-finite \(non-commutative\) $R$-algebra.
  Let $M$ be a finitely generated graded left $T$-module.
  Then the following are equivalent.
  \begin{enumerate}
  \item[\rm(1)] $\hat M$ is indecomposable as a $\hat T$-module, where $\widehat{(-)}$ denotes
    the $\fm$-adic completion.
  \item[\rm(2)] $M_\fm$ is indecomposable as a $T_\fm$-module.
  \item[\rm(3)] $M$ is indecomposable as a $T$-module.
  \item[\rm(4)] $M$ is indecomposable as a graded $T$-module.
  \end{enumerate}
\end{theorem}

\begin{proof}
  (1)$\Rightarrow$(2).
  If $M_\fm\cong N_1\oplus N_2$ as a $T_\fm$ module with $N_1\neq 0$ and $N_2\neq 0$, then
  taking the completion, $\hat M \cong \hat N_1\oplus \hat N_2$, and $\hat N_1\neq 0$ and $\hat N_2
  \neq 0$, and $\hat M$ is not indecomposable.
  If $M_\fm=0$, then $\hat M=0$, and $\hat M$ is not indecomposable.

  (2)$\Rightarrow$(3).
  If $M=0$, then $M_\fm=0$.
  Assume that
  there is a decomposition $M=M_1\oplus M_2$ with $M_1\neq 0$ and $M_2\neq 0$ in the category of $T$-modules.
  Let $m_1$ and $m_2$ be non-zero elements of $M_1$ and $M_2$, respectively.
  Then it is easy to see that there exists some $r\ge 1$ such that both $m_1$ and $m_2$ are nonzero in
  $M/\fm^rM$ by Lemma~\ref{Krull.lem}.
  Then $M_1/\fm^r M_1\neq 0$ and $M_2/\fm^r M_2\neq 0$.
  This shows $(M_1)_\fm\neq 0$ and $(M_2)_\fm\neq 0$.
  This contradicts the indecomposability of $M_\fm$, since $M_\fm=(M_1)_\fm\oplus (M_2)_\fm$.
  
  (3)$\Rightarrow$(4).
  Set $E=\End_TM$.
  Then $E\neq 0$, and $E$ does not have a non-trivial idempotent.
  So $E_0\neq 0$, and $E_0$ does not have a non-trivial idempotent.
  As the endomorphism ring of $M$ as an object of $T\grmod$ is $E_0$,
  we have that $M$ is indecomposable as an object of $T\grmod$.

  (4)$\Rightarrow$(1).
  Let $\hat e$ be an idempotent of $\hat E=\End_{\hat T}\hat M$.
  If $\hat e\in \hat J$, then $\hat e=0$ by Lemma~\ref{idempotent.lem}.
  If $1-\hat e\in \hat J$, then $1-\hat e=0$.
  So if $\hat e$ is nontrivial, then the image of $\hat e$ in $\hat E/\hat J\cong E_0/J_0$
  must be still nontrivial, but this is absurd, since $E_0/J_0$ is a division ring where
  there is no nontrivial idempotent.
\end{proof}

\begin{corollary}\label{main.cor}
  Let $M$ and $N$ be objects in $T\grmod$, the category of finitely generated graded left $T$-modules.
  Then the following are equivalent.
  \begin{enumerate}
  \item[\rm(1)] If $M=M_1\oplus\cdots\oplus M_s$ and $N=N_1\oplus\cdots\oplus N_t$ are
    decompositions into indecomposable objects in $T\grmod$, then $s=t$, and
    there exists some permutation $\sigma\in \frak S_s$ and integers $d_1,\ldots,d_s$
    such that $N_i\cong M_{\sigma i}(d_i)$, where $-(d_i)$ denotes the shift of degree.
  \item[\rm(2)] $M\cong N$ as $T$-modules.
  \item[\rm(3)] $M_\fm\cong N_\fm$ as $T_\fm$-modules.
    \item[\rm(4)] $\hat M\cong \hat N$ as $\hat T$-modules.
  \end{enumerate}
\end{corollary}

\begin{proof}
  (1)$\Rightarrow$(2)$\Rightarrow$(3)$\Rightarrow$(4) is trivial.
  We prove (4)$\Rightarrow$(1).
  As $\hat R$ is a Noetherian complete local ring and $\hat T$ is its module-finite algebra,
  the category of finite $\hat T$-modules is Krull--Schmidt.
  As $\hat M=\hat M_1\oplus\cdots\oplus \hat M_s$ and $\hat N=\hat N_1\oplus \cdots\oplus \hat N_t$
  are decompositions into indecomposable $\hat T$-modules by Theorem~\ref{main.thm},
  we have that $s=t$, and after change of indices (if necessary), there are isomorphisms
  $\hat M_i\cong \hat N_i$ for $i=1,\ldots,s$.
  So we may assume that $s=t=1$.

  Let $\hat\phi:\hat M\rightarrow \hat N$ be a $\hat T$-isomorphism, and let
  $\hat \psi:\hat N\rightarrow \hat M$ be its inverse.
  We can write $\hat\phi=\sum_i\hat a_i\phi_i$ and $\hat \psi=\sum_j\hat b_j\psi_j$ with
  $\phi_i\in\Hom_T(M,N)_{u_i}$, $u_i\in\ZZ$, $\hat a_i\in\hat R$,
  $\psi_j\in\Hom_T(N,M)_{v_j}$, $v_j\in\ZZ$, and $\hat b_j\in\hat R$.
  So $1_{\hat M}=\hat \psi\hat\phi=\sum_{ij}\hat a_i\hat b_j\psi_j\phi_i$.
  As $1_{\hat M}\notin \hat J$, there exists some $(i,j)$ such that
  $\psi_j\phi_i\notin J$, where $J={}^*\rad\End_TM$.
  As $\psi_j\phi_i$ is a homogeneous element of $E=\End_TM$, which is ${}^*$local,
  we have that $\psi_j\phi_i$ is a unit of $E$.
  In particular, $\psi_j$ is a split epimorphism.
  As $N$ is also indecomposable, we have that $\psi_j$ is a $T$-isomorphism.
  So we have that $\psi_j:N\rightarrow M(v_j)$ is an isomorphism in $T\grmod$.
\end{proof}
  
\begin{corollary}\label{main2.cor}
  Let $R=\bigoplus_{i\geq 0}R_i$ be a Noetherian $\ZZ_{\geq 0}$-graded commutative ring such that
  $(R_0,\fm_0)$ is an $F$-finite Henselian local ring of prime characteristic $p$.
  Let $\fm=\fm_0+R_+$, where $R_+=\bigoplus_{i>0}R_i$.
  Let $\hat R$ be the $\fm$-adic completion of $R$.
  Let $M_1,\ldots,M_r$ be finitely generated $\QQ$-graded $R$-modules.
  Then the following are equivalent.
  \begin{enumerate}
  \item[\rm(1)] $R$ has finite $F$-representation type \(FFRT for short\) in the graded sense with $M_1,\ldots,M_r$.
    That is,
    \begin{enumerate}
      \item[\rm(1-a)] For each $i$, $M_i$ is indecomposable;
    \item[\rm(1-b)] For each $i$, there exists some $e\geq 0$ and $c\in\QQ$ such that $M_i(c)$ is a direct summand of 
      ${}^eR$;
    \item[\rm(1-c)] For each $e\geq 0$, any indecomposable direct summand of ${}^eR$ is
        isomorphic to $M_i(c)$ for some $1\leq i\leq r$ and $c\in\QQ$.
    \end{enumerate}
  \item[\rm(2)] The local ring $\hat R$ has FFRT with $\hat M_1,\ldots,\hat M_r$.
    That is,
\begin{enumerate}
\item[\rm(2-a)] For each $i$, $\hat M_i$ is indecomposable;
\item[\rm(2-b)]  For each $i$, there exists some $e\geq 0$ such that $\hat M_i$ is a direct summand of ${}^e\hat R$;
\item[\rm(2-c)] For each $e\geq 0$, any indecomposable direct summand of ${}^e\hat R$ is
  isomorphic to $\hat M_i$ for some $1\leq i\leq r$.
\end{enumerate}
  \end{enumerate}
  In particular, $R$ has FFRT in the graded sense if and only if $\hat R$ has FFRT.
\end{corollary}

\begin{proof}
  By Theorem~\ref{main.thm}, (1-a) and (2-a) are equivalent.

  Note that for each $e$,
  \[
    {}^e\hat R\cong\projlim {}^e(R/\fm^n)
    =
    \projlim {}^eR/{}^e(\fm^n)
    =
    \projlim {}^eR/{}^e((\fm^n)^{[p^e]})
    =
    \projlim {}^eR / (\fm^n ){}^eR=\widehat{{}^eR},
    \]
    where $\fm^{[p^r]}$ is the ideal of $R$ generated by $\{a^{p^r}\mid a\in \fm\}$.

    (1-a)$\Rightarrow$(2-a).
    If $M_i(c)$ is a direct summand of ${}^eR$, then $\hat M_i$ is a direct summand of
    $\widehat{{}^e R}={}^e \hat R$.
    
    (2-a)$\Rightarrow$(1-a).
    Assume that $\hat M_i$ is a direct summand of ${}^e\hat R$.
    If we have a decomposition
    \begin{equation}\label{e-dec.eq}
      {}^eR \cong N_1\oplus\cdots\oplus N_r,
    \end{equation}
      where each $N_j$ is an indecomposable $\QQ$-graded $R$-modules,
      then
      \begin{equation}\label{e-dec-hat.eq}
        {}^e\hat R\cong\widehat{{}^eR}\cong \hat N_1\oplus\cdots\oplus \hat N_r.
      \end{equation}
      This is the decomposition into indecomposable modules by Theorem~\ref{main.thm}.
      So by the Krull--Schmidt property of the category of $\hat R$-modules,
      $\hat M_i\cong \hat N_j$ for some $j$.
      By Corollary~\ref{main.cor}, $N_j\cong M_i(c)$ for some $c\in\QQ$.

      (1-c)$\Rightarrow$(2-c).
      Let $\hat N$ be a direct summand of ${}^e\hat R$.
      By assumption, there is a decomposition (\ref{e-dec.eq}) such that each $N_j$ is
      isomorphic to $M_{i(j)}(c_j)$ for some $1\leq i(j)\leq r$ and $c_j\in\QQ$.
      Then the isomorphism (\ref{e-dec-hat.eq}) holds, and by the Krull--Schmidt property,
      $\hat N\cong \hat N_j\cong \hat M_{i(j)}$.

      (2-c)$\Rightarrow$(1-c).
      Let $L$ be a direct summand of ${}^eR$.
      Then $\hat L$ is a direct summand of ${}^e\hat R$.
      So $\hat L\cong \hat M_i$ for some $i$.
      Hence by Corollary~\ref{main.cor}, $L\cong M_i(c)$ for some $c\in\QQ$.

      We prove the last assertion.
      The \lq only if' part is clear from what we have proved above.
      We prove the \lq if' part.
      Let $\hat R$ have FFRT with the finite indecomposable $\hat R$-modules
      $\hat L_1,\ldots,\hat L_r$.
      So for each $i$, $\hat L_i$ is a direct summand of ${}^e\hat R$ for some $e$.
      Now let (\ref{e-dec.eq}) be the decomposition of ${}^eR$ into indecomposable $\QQ$-graded $R$-modules.
      Then we have an isomorphism (\ref{e-dec-hat.eq}), which is a decomposition into indecomposables by
      Theorem~\ref{main.thm}.
      By the Krull--Schmidt, $\hat L_i\cong \hat N_j$ for some $j$.
      In particular, there is a finite indecomposable $\QQ$-graded module $M_i$ such that $\hat M_i\cong \hat L_i$.
      By what we have proved above, $R$ has FFRT with $M_1,\ldots,M_r$ in the
      graded sense, as required.
\end{proof}

\begin{flushleft}
Mitsuyasu Hashimoto\\
Department of Mathematics\\
Osaka Metropolitan University\\
Sumiyoshi-ku, Osaka 558--8585, JAPAN\\
e-mail: {\tt mh7@omu.ac.jp}
\end{flushleft}

\begin{flushleft}
  Yuntian Yang\\
  Saitama, JAPAN\\
e-mail: {\tt yyt0317@hotmail.co.jp}
\end{flushleft}

\end{document}